\documentclass[12pt]{amsart}
\usepackage{amscd}
\usepackage{amsmath, amsfonts,amssymb,amsthm,mathrsfs,xypic}
\usepackage{graphicx}
\xyoption{curve}

\usepackage{fancybox}
\usepackage[all]{xy}
\numberwithin{equation}{section}

\theoremstyle{plain}
\newtheorem{theorem}[equation]{Theorem}

\newtheorem{lemma}[equation]{Lemma}
\newtheorem{corollary}[equation]{Corollary}
\newcommand{\Hom}{\operatorname{Hom}}
\theoremstyle{definition}
\newtheorem{definition}[equation]{Definition}
\newtheorem{example}[equation]{Example}

\theoremstyle{remark}

\newcommand{\xrto}{\xrightarrow}

\def\C{\mathcal{C}}

\def\G{\mathcal{G}}

\def\P{\mathcal{P}}

\def\D{\mathcal{D}}
\def\F{\mathcal{F}}

\def\A{\mathcal{A}}
\def\T{\mathcal{T}}
\def\X{\mathcal{X}}
\def\Y{\mathcal{Y}}

\begin{document}

\title [\tiny{Complete cotorsion pairs in exact categories}]{Complete cotorsion pairs in exact categories }
\author [\tiny{Zhi-Wei Li}] {Zhi-Wei Li}

\date{\today}
\thanks{The author was supported by National Natural Science Foundation
of China (No.s 11671174 and 11571329).}
\address{Zhi-Wei Li\\ School of Mathematics and Statistics \\
 Jiangsu Normal University,
Xuzhou 221116, Jungian, PR China.}
\email{zhiweili@jsnu.edu.cn}
\date{\today}


\maketitle

\begin{abstract} We show a cotorsion pair cogenerated by a class is complete under suitable conditions in an arbitrary exact category using the generalized small object argument given by Chorny. This recovers Saor\'in and \v{S}\v{t}ov\'{i}\v{c}ek's criterion of the completeness of cotorsion pairs in their efficient exact categories. Examples in the categories of chain complexes of exact categories are given.
\end{abstract}

\setcounter{tocdepth}{1}

\section{Introduction}
Ever since Salce introduced the notion of a cotorsion pair in the late 1970's \cite{Salce}, the significance of complete cotorsion pairs has been widely understood in approximation theory \cite{Gobel/Trlifaj}, especially in the proof of the flat cover conjecture \cite{BBE}. 

  One fundamental result on completeness of cotorsion pairs is due to Eklof and Trlifaj, they proved that any cotorsion pair cogenerated by a set of modules is complete \cite{ET}, which is a generalization of the corresponding result of G\"obel and Shelah on abelian groups in \cite{Gobel/Shelah}. Hovey extended this result to Grothendieck categories \cite{Hovey02}, and later, Saor\'in and \v{S}\v{t}ov\'{i}\v{c}ek generalized it further to their efficient exact categories \cite{Saorin/Stovicek}.

However, the condition that inflations are closed under transfinite compositions for an efficient exact category is somewhat ironically. For example, for a Grothendieck category $\G$ with a generator $G$, the category ${\rm Ch}(\G_G)$ of chain complexes of the $G$-exact category $\G_G$ considered in \cite{Gillespie14} is not necessarily efficient if $G$ is not finitely presented, even it is bicomplete. 

On the other hand, for the category ${\rm Ch}(\A)$ of chain complexes of a bicomplete abelian category $\A$ determined by a projective class $\P$, Christensen and Hovey constructed a relative closed model structure by defining weak equivalences as the $\P$-equivalences and fibrations the $\P$-fibrations \cite{Christensen/Hovey}. Apparently, we can define a $\P$-exact category $\A_\P$ where a conflation is a short exact sequence in $\A$ which is $\P$-exact, and we want to know whether the relative closed model structure on ${\rm Ch}(\A)$ is just the standard projective exact closed model structure on ${\rm Ch}(\A_\P)$. But this relates the completeness of cotorsion pairs cogenerated by classes (not necessarily sets).

This paper aims to give a criterion of the completeness of cotorsion pairs cogenerated by classes in an arbitrary exact categories. The basic tool is the generalized small object argument given by Chorny in \cite{Chorny}. Our main result is the following.

\vskip5pt
\noindent{\bf Theorem} (\ref{thm:ccp}) {\it  Let $\A$ be an exact category. Let $I$ be a homological class of inflations which permits the generalized small object argument. Denote $\F=({\rm Cok}(I))^{\perp}$.

$(\mathrm{i})$ If each relative $I$-cell complex with codomain in $\F$ is an inflation and ${^\perp\F}$ is a class of generators, then $({^\perp\F}, \F)$ is a complete cotorsion pair in $\A$.

$(\mathrm{ii})$ If each morphism in $I$-inj with domain in ${\rm Cell}(I)$ is a deflation and $\F$ is a class of cogenerators, then $({\rm Cof}(I), \F)$ is a complete cotorsion pair in $\A$.

$(\mathrm{iii})$ If each relative $I$-cell complex with codomain in $\F$ is an inflation, ${\rm Cof}(I)$ is a class of generators and $\A$ is WIC, then $({\rm Cof}(I), \F)$ is a complete cotorsion pair in $\A$.

$(\mathrm{iv})$ If each relative $I$-cell complex with codomain in $\F$ is an inflation and each morphism in $I$-inj with domain in ${\rm Cell}(I)$ is a deflation, then $({\rm Cof}(I), \F)$ is a complete cotorsion pair in $\A$.}
\vskip5pt

The contents of the paper are as follows: In Section 2, we define the necessary notation, explain the notion of a class of morphisms which permits the generalized small object argument and restate Chorny's generalized small object argument. Section 3 is devoted to the proof of our main result. In the last Section, we give some examples 
of complete cotorsion pairs in the categories of complexes of (relative) exact categories. 
\vskip5pt
Throughout this paper, all colimits in concern are small colimits.

\subsection*{Acknowledgements}  The author would like to thank Henning Krause, Xiao-Wu Chen, Jan \v{S}\v{t}ov\'{i}\v{c}ek and Guodong Zhou for their helpful discussions and comments.

\section{The generalized small object argument}

In this section, we give the generalized small object argument which is essentially proved by Chorny in \cite{Chorny}. We follow the sequence of lemmas from \cite{Hovey99}, extending them to work for the general case. The proofs are essentially the same but we include the general versions here for clarity and convenience of the reader.

.
\subsection{Relative $I$-cell complexes}
Let $\C$ be a category. Suppose $i\colon A\to B$ and $p\colon X\to Y$ are morphisms in $\C$. Given a morphism $(f, g)\colon i\to p$, ie, a commutative diagram in $\C$ of the following form
\[\xymatrixcolsep{2pc}\xymatrix@C14pt@R14pt{A\ar[r]^f\ar[d]_{i}&
X\ar[d]^{p}\\
B \ar[r]^{g} \ar@{.>}[ur]^{h}& Y}
\]
a {\it lift} or {\it lifting} in the diagram is a morphism $h\colon B\to X$ such that $hi=f$ and $ph=g$.
A morphism $i\colon A\to B$ is said to have the {\it left lifting property} (LLP) with respect to another morphism $p\colon X\to Y$ and $p$ is said to have the {\it right lifting property} (RLP) with respect to $i$ if a lift exists in any diagram of the above form.
\begin{definition} (\cite[Definition 2.1.7]{Hovey99}) Let $I$ be a class of morphisms in a category $\C$.

$(1)$\ A morphism is {\it $I$-injective} if it has the RLP with respect to every morphism in $I$. The class of $I$-injective morphisms is denoted $I$-inj.

$(2)$ \ A morphism is an {\it $I$-cofibration} if it has the LLP with respect to every $I$-injective morphism. The class of $I$-cofibrations is denoted $I$-cof.
\end{definition}

If $\C$ has an initial object $0$, an object $A\in \C$ is {\it $I$-cofibrant} if the morphism $0\to A\in I\mbox{-cof}$. The collection of $I$-cofibrants is denoted ${\rm Cof}(I)$.

Suppose $\C$ is a category and $\lambda$ is an ordinal. A functor $X\colon \lambda \to \C$ (ie, a diagram
$$X_0\to X_1\to X_2\to \cdots \to X_\alpha\to \cdots  (\alpha<\lambda)$$
in $\C$) is called a {\it $\lambda$-sequence} if for every limit ordinal $\gamma<\lambda$ the colimit ${\rm colim}_{\alpha<\gamma}X_\alpha$ exists and the induced morphism ${\rm colim}_{\alpha<\gamma}X_\alpha\to X_\gamma$ is an isomorphism.

If a colimit of a $\lambda$-sequence $X$ exists, the morphism $X_0\to {\rm colim}_{\alpha<\lambda}X_\alpha$ is called the {\it transfinite composition} of $X$.

If $\D$ is a collection of morphisms in a category $\C$ and $\lambda$ is an ordinal, a {\it $\lambda$-sequence of morphisms in $\D$} is a $\lambda$-sequence $X_0\to X_1\to X_2\to \cdots \to X_\alpha\to \cdots  (\alpha<\lambda)$ in $\C$ such that each morphism $X_\alpha\to X_{\alpha+1}$ is in $\D$ for $\alpha+1<\lambda$. Then a {\it transfinite composition of morphisms in $\D$} is the transfinite composition of a $\lambda$-sequence $X_0\to X_1\to X_2\to \cdots \to X_\alpha\to \cdots  (\alpha<\lambda)$ of morphisms in $\D$.

\begin{definition}  $($\cite[Definition 2.1.9]{Hovey99}$)$  Let $I$ be a class of morphisms in a category $\C$. Assume that the transfinite compositions of pushouts of morphisms in $I$ exist. A {\it relative $I\mbox{-}cell$ complex} is a transfinite composition of pushouts of morphisms in $I$.
\end{definition}

The collection of relative $I$-cell complexes is denoted $I\mbox{-cell}$. Note that $I\mbox{-cell}$ contains all isomorphisms. If $\C$ has an initial object $0$, an object $A\in \C$ is an {\it $I$-cell complex} if the morphism $0\to A\in I\mbox{-cell}$. The collection of $I$-cell complexes is denoted ${\rm Cell}(I)$.

\begin{lemma} \label{lem:cellcof} Let $\C$ be an arbitrary category and $I$ be a class of morphisms in $\C$. If the transfinite compositions of pushouts of morphisms in $I$ exist, then $I\mbox{-}{\rm cell}\subseteq I\mbox{-}{\rm cof}$.
\end{lemma}
\begin{proof}  Assume that we have a commutative diagram
\[\xymatrixcolsep{2pc}\xymatrix@C10pt@R10pt{C\ar[r]^{g}\ar[d]_{f} &
A\ar[d]^{j}\\
D \ar[r]^{h}& B}
\]
such that $j\in I$-inj and $f$ is a relative $I$-cell complex.

Let $f\colon C\to D$ be the transfinite composition of the $\lambda$-sequence $C=X_0\xrto{f_0} X_1\xrto{f_1} X_2\cdots\to X_\alpha\xrightarrow{f_\alpha} X_{\alpha+1}\to \cdots  (\alpha+1<\lambda)$
with the $f_\alpha$ being a pushout of a morphism in $I$. Let $\tau_\alpha\colon X_\alpha \to D={\rm colim}_{\alpha<\lambda}X_\alpha$ be the colimit morphism for all $\alpha<\lambda$. We will construct the morphism $u_\alpha\colon X_\alpha\to A$ by transfinite induction such that $ju_\alpha=h\tau_\alpha$ and $u_{\alpha+1}f_{\alpha}=u_\alpha$. Let $u_0=g$. Assume that $X_0\to X_1$ is the pushout of $i\colon E\to F$ in $I$:

\[\xymatrixcolsep{2pc}\xymatrix@C28pt@R18pt{E\ar[r]\ar[d]_{i}&
X_0\ar[r]^{g}\ar[d]^{f_0} & A \ar[d]^{j}\\
F \ar[r]\ar@{.>}@/_10pt/[urr]|(.35){v}& X_1 \ar[r]^{h \tau_1} & B.}
\]
Since $j\in I\mbox{-}{\rm inj}$, there is a lifting $v\colon F\to A$ which induces a morphism $u_1\colon X_1\to A$ such that $ju_1=h\tau_1$ and $u_1f_0=g$ by the universal property of pushouts. Assume that we have defined $u_\alpha\colon X_\alpha\to A$ for all $\alpha<\beta$. If $\beta$ is a limit ordinal, let $u_\beta\colon X_\beta={\rm colim}_{\alpha<\beta}X_\alpha \to A$ be the induced morphism by $u_\alpha$ for $\alpha<\beta$, then $ju_\beta=h\tau_\beta$. If $\beta$ has a predecessor $\alpha$, then replace $f_0\colon X_0\to X_1$ by $f_\alpha$ in the case of $\alpha=0$, we can construct a morphism $u_{\alpha+1}\colon X_{\alpha+1}\to A$ satisfying $u_{\alpha+1}f_\alpha= u_\alpha$ and $ju_{\alpha+1}=h\tau_{\alpha+1}$ which completes our transfinite induction. Therefore, let $u\colon D={\rm colim}_{\alpha<\lambda}X_\alpha\to A$ be the induced morphism by the $u_\alpha$, then $ju=h$ and $ugh=g$ by the universal property of colimits.
\end{proof}

\begin{lemma} \label{le:transcell} Let $\C$ be a category and $I$ a class of morphisms in $\C$. If the transfinite compositions of pushouts of morphisms in $I$ exist. Then the transfinite compositions of $I\mbox{-}\mathrm{cell}$ exist and are in $I\mbox{-}\mathrm{cell}$.
\end{lemma}

\begin{proof}  Let $\lambda$ be an ordinal and $X$ be a $\lambda$-sequence $X_0\to X_1\to X_2\to \cdots \to X_\alpha\to \cdots (\alpha<\lambda)$
such that each morphism $X_\alpha\to X_{\alpha+1}$ for $\alpha+1<\lambda$ is the transfinite composition of the $\gamma_\alpha$-sequence $X_\alpha=W_0^\alpha\to W_1^\alpha\to W_2^\alpha\to \cdots \to W_\beta^\alpha\to \cdots (\beta<\gamma_\alpha)$
of pushouts of morphisms in $I$. By {\it interpolating} these sequences for all $\alpha<\lambda$ into the $\lambda$-sequence $X$ \cite[Definition 10.2.11]{Hirschhorn03}, we get a $\mu$-sequence $Y\colon \mu\to \C$ of pushouts of morphisms in $I$ \cite[Propositions 10.2.8 and 10.2.13]{Hirschhorn03}. By assumption, the transfinite composition of the $\mu$-sequence $Y$ exists, that is, ${\rm colim}_{\gamma<\mu}Y_\gamma$ exists. By the construction of $Y$, we have ${\rm colim}_{\alpha<\lambda}X_\alpha={\rm colim}_{\gamma<\mu}Y_\gamma$ and the transfinite composition $X_0\to {\rm colim}_{\alpha<\lambda}X_{\alpha}$ is the transfinite composition $Y_0\to {\rm colim}_{\gamma<\mu}Y_{\gamma}$ which is a relative $I$-cell complex.
\end{proof}

\subsection{The generalized small object argument}

Recall that, the {\it cofinality} of a limit ordinal $\lambda$, denoted by ${\rm cf}(\lambda)$, is the smallest cardinal $\kappa$ such that there exists a subset $T$ of $\lambda$ with $|T|=\kappa$ and ${\rm sup}(T)=\lambda$.

Let $\C$ be a category. Let $\kappa$ be a cardinal and $\D$ a class of morphisms in $\C$. An object $A$ of $\C$ is said to be {\it $\kappa$-small relative to $\D$} if for every ordinal $\lambda$ with ${\rm cf}(\lambda)>\kappa$ and every $\lambda$-sequence $X\colon \lambda\to \C$ of morphisms in $\D$, the natural morphism $${\rm colim}_{\alpha<\lambda}\Hom_\C(A, X_\alpha)\to \Hom_\C(A, {\rm colim}_{\alpha<\lambda}X_\alpha)$$
is an isomorphism. An object $A$ in $\C$ is called {\it small relative to $\D$} if it is $\kappa$-small relative to $\D$ for some cardinal $\kappa$.

\begin{definition}\label{def:soa}  Let $\C$ be a category. We say a class $I$ of morphisms in $\C$ {\it permits the generalized small object argument} if the following conditions hold:

(i) The transfinite compositions of pushouts of morphisms in $I$ exist.

(ii)  There is a cardinal $\kappa$, such that the domains of the morphisms of $I$ are $\kappa$-small relative to $I$-cell.

(iii) For every morphism $f$ in $\C$, there is a morphism $g\in I\mbox{-}{\rm cell}$ equipped with a morphism $s\colon g\to f$, such that any morphism $i\to f$ with $i\in I$ factors through $s$.
\end{definition}

\begin{lemma} \label{lem:po} Let $I$ be a class of morphisms of a category $\C$ such that the transfinite compositions of pushouts of morphisms in $I$ exist. Then any pushouts of morphisms in $I\mbox{-}\mathrm{cell}$ exist and are in $I\mbox{-}\mathrm{cell}$.

\end{lemma}
\begin{proof} Assume that $f\colon A\to B$ is a relative $I$-cell complex. Then there is an ordinal $\lambda$ and a $\lambda$-sequence $A=X_0\xrto{f_0} X_1\xrto{f_1} \cdots \to X_\beta\xrto{f_\beta} X_{\beta+1}\to  \cdots  (\beta+1<\lambda)$ such that every $f_\beta$ is a pushout of a morphism in $I$ and $f$ is the transfinite composition of $X$. Let
\[\xymatrixcolsep{2pc}\xymatrix@C10pt@R10pt{A\ar[r]^{f}\ar[d]_{g_0}&
B\\
E_0 &}
\]
be any diagram in $\C$.

We can construct a commutative diagram by transfinite induction
\[
\xymatrixcolsep{2pc}\xymatrix@C10pt@R10pt{X_0\ar[r]^{f_0}\ar[d]_{(*)\qquad}_{g_0}& X_1\ar[r]^{f_1} \ar[d]_{g_1} & \cdots \ar[r]& X_{\beta} \ar[d]_{g_{\beta}} \ar[r]^{f_\beta} & \cdots \\
E_0 \ar[r]^{h_0}& E_1\ar[r]^{h_1}&  \cdots \ar[r] & E_\beta \ar[r]^{h_\beta} & \cdots}
\]
such that each square diagram is a pushout. In fact, since $f_0$ is a pushout of a morphism in $I$, say $i\colon C\to D$, we have a pushout diagram
\[\xymatrixcolsep{2pc}\xymatrix@C10pt@R10pt{C\ar[r]^{s}\ar[d]_{i}& X_0\ar[d]^{f_0}\\
D \ar[r] & X_1}
\]
By assumption, the pushout of $i$ along $g_0s$ exists:
\[\xymatrixcolsep{2pc}\xymatrix@C10pt@R10pt{C\ar[r]^{g_{0}s}\ar[d]_{i}& E_0\ar[d]^{h_0}\\
D \ar[r] & E_1}
\]
which induces a pushout diagram
\[\xymatrixcolsep{2pc}\xymatrix@C10pt@R10pt{X_0\ar[r]^{f_0}\ar[d]_{g_0}& X_1 \ar[d]^{g_1}\\
E_0 \ar[r]^{h_0} & E_1}
\]
 Assume that we have defined $E_\alpha$ and $g_\alpha\colon X_\alpha\to E_\alpha$ for all $\alpha<\beta$. If $\beta$ is a limit ordinal, define $E_\beta={\rm colim}_{\alpha<\beta} E_\alpha$ which exists by assumption, and define $g_\beta$ to be the morphism induced by $g_\alpha$. If $\beta$ has a predecessor $\alpha$, ie, $\beta=\alpha+1$, define $E_\beta=E_{\alpha}\coprod_{X_\alpha}X_{\alpha+1}$, and $g_{\beta}$ to be the pushout of $g_\alpha$ along $f_\alpha$. Therefore we have a $\lambda$-sequence $E_0\xrto{h_0} E_1\xrto{h_1} \cdots E_\beta \xrto{h_\beta} E_{\beta+1} \to \cdots  (\beta+1<\lambda),$
  its transfinite composition $E_0\to {\rm colim}_{\alpha<\lambda}E_\alpha$ exists by assumption and is in $I$-cell by construction. The commutative diagram $(*)$ induces a desired pushout diagram
  \[\xymatrixcolsep{2pc}\xymatrix@C10pt@R10pt{A\ar[r]^-{f}\ar[d]_{g_0} & B={\rm colim}_{\alpha<\lambda}X_\alpha \ar[d]\\
E_0 \ar[r] & {\rm colim}_{\alpha<\lambda}E_\alpha}
\]
in $\C$. \end{proof}
Now we can prove the following generalized Quillen small object argument, which is essentially given in \cite[Theorem 1.1]{Chorny}.

\begin{theorem} $(${\rm The generalized small object argument}$)$.\ Let $\C$ be an arbitrary category and $I$ a class of morphisms in $\C$. Suppose that $I$ permits the generalized small object argument. Then every morphism $f\colon X\to Y$ in $\C$ admits a factorization $f=\delta(f)\gamma(f)$, where $\gamma(f)\in I\mbox{-}\mathrm{cell}$ and $\delta(f)\in I\mbox{-}\mathrm{inj}$.
\end{theorem}
\begin{proof} By Lemmas 2.4 and 2.6, the proof of \cite[Theorem 1.1]{Chorny} works here.
\end{proof}

\section{Complete cotorsion pairs in exact categories}
In this section, we recall the definition of complete cotorsion pairs in exact categories and prove our main result. 

\subsection{Cotorsion pairs in exact categories }
The concept of an exact category is due to D. Quillen \cite{Quillen73}, a simple axiomatic description can be found in \cite[Appendix A]{Keller90}.  Roughly speaking {\it an exact category } is an additive category $\A$ equipped with a class $\mathcal{E}$ of {\it kernel-cokernel sequences} $A\xrto{s} B\xrto{t} C$ in $\A$ such that $s$ is the kernel of $t$ and $t$ is the cokernel of $s$. The class $\mathcal{E}$ satisfies exact axioms, for details, we refer the reader to \cite[Definition 2.1]{Buhler10}. Given an exact category $\A$, we will call a kernel-cokernel sequence a {\it conflation} if it is in $\mathcal{E}$. The morphism $s$ in a conflation $A\xrto{s} B\xrto{t} C$  is called an {\it inflation} and the morphism $t$ is called a {\it deflation}.

Given an exact category $\A$, recall that a pair $(\X, \Y)$ of classes of objects of $\A$ is called a {\it cotorsion pair} in $\A$ if $\X={^{\perp}\Y}:=\{X\in \A \ | \ {\rm Ext}^1_\A(X, Y)=0, \forall \ Y \in \Y\}$ and $\Y={\X^{\perp}}:=\{Y\in \A \ | \ {\rm Ext}^1_\A(X, Y)=0, \forall \ X\in \X\}.$ The cotorsion pair $(\X, \Y)$ is called {\it complete} if for each $A\in \A$ there exist a conflation $ Y\to X\to A$ and $ A\to Y'\to X'$ such that $X,X'\in \X$ and $Y, Y'\in \Y$.

 A cotorsion pair $(\X, \Y)$ is said to be {\it cogenerated by a class} if there is a class $\mathcal{S}$ of objects in $\A$ such that $\Y=\mathcal{S}^{\perp}$.

\subsection{Eklof Lemma}

 Recall that, given a class $\mathcal{S}$ of objects in an exact category $\A$, an object $A$ of $\A$ is called a {\it transfinite extension of $\mathcal{S}$} if the morphism $0\to A$ is the transfinite composition of a $\lambda$-sequence $X_0\xrto{i_0} X_1\xrto{i_1} \cdots \to X_\beta\xrto{i_\beta} X_{\beta+1}\to  \cdots  (\beta+1<\lambda)$
such that each $i_\beta$ is an inflation with a cokernel in $\mathcal{S}$.

\vskip5pt
The following lemma is proved in \cite[Lemma 6.2]{Hovey02}, called {\it Eklof Lemma}.

\begin{lemma}  \label{lem:Ekloflem} Let $\A$ be an exact category and $A\in \A$. Then ${^{\perp}A}$ is closed under transfinite extensions and retracts.
\end{lemma}

\subsection{Completeness of cotorsion pairs}

Given a class $I$ of inflations in an exact category $\A$, let ${\rm Cok}(I)=\{A\in \A \ | \ A\cong {\rm coker}(i),\ \mbox{for some} \ i\in I\}$.
Following \cite[Definition 2.3]{Saorin/Stovicek}, $I$ is {\it homological} if for any object $A\in \A$, the morphism $A\to 0$ belongs to $I\mbox{-inj}$ if and only if $A\in {\rm Cok}(I)^{\perp}$.

Let $\A$ be an exact category. A collection $\T$ of objects of $\A$ is called a {\it class of generators} of $\A$ if for any object $A\in \A$, there is a deflation $\coprod_{s\in S}G_s\to A$ with $G_s\in \T$ and $S$ a set. Dually, $\T$ is called a {\it a class of cogenerators} of $\A$ if for any object $A\in \A$, there is an inflation $A\to \prod_{s\in S}G_s$ with $G_s\in \T$ and $S$ a set.

Recall that an exact category is called {\it weakly idempotent complete} ({\it WIC}, for short) if every split monomorphism is an inflation.

The following is our main result.

\begin{theorem}  \label{thm:ccp} Let $\A$ be an exact category. Let $I$ be a homological class of inflations which permits the generalized small object argument. Denote $\F=({\rm Cok}(I))^{\perp}$.

$(\mathrm{i})$ If each relative $I$-cell complex with codomain in $\F$ is an inflation and ${^\perp\F}$ is a class of generators, then $({^\perp\F}, \F)$ is a complete cotorsion pair in $\A$.

$(\mathrm{ii})$ If each morphism in $I$-inj with domain in ${\rm Cell}(I)$ is a deflation and $\F$ is a class of cogenerators, then $({\rm Cof}(I), \F)$ is a complete cotorsion pair in $\A$.

$(\mathrm{iii})$ If each relative $I$-cell complex with codomain in $\F$ is an inflation, ${\rm Cof}(I)$ is a class of generators and $\A$ is WIC, then $({\rm Cof}(I), \F)$ is a complete cotorsion pair in $\A$.

$(\mathrm{iv})$ If each relative $I$-cell complex with codomain in $\F$ is an inflation and each morphism in $I$-inj with domain in ${\rm Cell}(I)$ is a deflation, then $({\rm Cof}(I), \F)$ is a complete cotorsion pair in $\A$.
\end{theorem}
\begin{proof} (i) Firstly, for each $A\in \A$, we claim that there is a conflation $A\to B\to C$ such that $B\in \F$ and $C\in {^{\perp}\F}$.
In fact, since $I$ permits the generalized small object argument, we can factor $A\to 0$ as the composition $A\xrto{f} B\to 0$ with $f\in I\mbox{-}{\rm cell}$ and $B\to 0\in I\mbox{-}{\rm inj}$. By assumption, $f$ is an inflation and $B\in \F$. Since $I\mbox{-}{\rm cell}$ is closed under pushouts by Lemma \ref{lem:po}, we know that $C={\rm coker}f\in {\rm Cell}(I)$. Thus $C$ is a transfinite extension of ${\rm Cok}I$ and therefore is in ${^\perp\F}$ by Lemma \ref{lem:Ekloflem}.

Secondly, for each $A\in \A$, we claim that there is a conflation $T\to U\to A$ such that $T\in \F$ and $U\in {^\perp\F}$. In fact, since ${^\perp\F}$ is closed under coproducts, and by assumption, it is a class of generators of $\A$, there exists a deflation $p\colon G\to A$ with $G\in {^\perp\F}$. Let $K={\rm ker}p$. By the first claim, there is a conflation $K\to T \to C$ with $T\in \F$ and $C\in {^\perp\F}$. By \cite[Proposition 2.12]{Buhler10}, there is a commutative diagram of conflations:
\[\xymatrixcolsep{2pc}\xymatrix@C10pt@R10pt  {K  \ar[r] \ar[d]&
T  \ar[r]\ar[d] & C \ar@{=}[d]\\
G  \ar[r] \ar[d]_{p} & U \ar[d]
\ar[r]& C  \\
A \ar@{=}[r] & A &}
\]
such that the upper-left square is a pushout diagram. Since ${^\perp\F}$ is closed under extensions, we know that $U\in {^\perp\F}$. Thus $T\to U\to A$ is the desired conflation. Therefore $({^\perp \F}, \F)$ is a complete cotorsion pair.

(ii) Firstly, for each $A\in \A$, we claim that there is a conflation $K\to U\to A$ such that $K\in \F$ and $U\in {\rm Cell}(I)$.
In fact, since $I$ permits the generalized small object argument, we can factor $0\to A$ as the composition $0\to U\xrto{f} A$ with $U\in{\rm Cell}(I)$ and $f\in I\mbox{-}{\rm inj}$. Since ${^\perp \F}$ is closed under transfinite extensions by Lemma \ref{lem:Ekloflem}, we know that $U\in {^\perp \F}$. By assumption, $f$ is a deflation. Since $I\mbox{-}{\rm inj}$ is closed under pullback, we know that $K=\ker f\to 0\in I\mbox{-}{\rm inj}$ which means $K\in \F$.

Secondly, for each $A\in \A$, we claim that there is a conflation $A\to F\to C$ such that $F\in \F$ and $C\in {^\perp\F}$. In fact, this can be proved dually with the proof of the second claim of (i).

Finally, we claim that ${\rm ^\perp \F}={\rm Cof}(I)$. In fact, assume that $M\in {\rm Cof}(I)$, then $0\to M\in I\mbox{-}{\rm cof}$. Using the generalized small object, we can factor $0\to M$ as the composition $0\to N\xrto{s} M$ with $N\in {\rm Cell}(I)$ and $s\in I\mbox{-}{\rm inj}$. Thus there is a morphism $t\colon M\to N$ such that $st=1_M$, ie, $M$ is a retract of $N$. Therefore $M\in {^{\perp}\F}$ since ${^{\perp}\F}$ is closed under retracts by Lemma \ref{lem:Ekloflem}. So ${\rm Cof}(I)\subseteq{\rm ^\perp \F}$. Conversely, assume that $A\in {^\perp \F}$, then the conflation $K\to U\to A$ in the first claim splits. Therefore $A$ is a direct summand of $U$. Since ${\rm Cell}(I)\subseteq {\rm Cof}(I)$ and ${\rm Cof}(I)$ is closed under retracts, we know that $A\in {\rm Cof}(I)$. So ${\rm ^\perp \F}\subseteq {\rm Cof}(I)$, and then ${\rm ^\perp \F}={\rm Cof}(I)$.

(iii) We claim that each morphism $f\colon A\to B$ in $I\mbox{-}{\rm inj}$ with $A\in {\rm Cell}(I)$ is a deflation. In fact, since ${\rm Cof}(I)$ is a class of generators and is closed under coproducts, there is a deflation $p\colon G\to B$ with $G\in {\rm Cof}(I)$. So there is a lift $g$ such that $p=fg$. Thus $f$ is a deflation by the obscure axiom \cite[Proposition 7.6]{Buhler10}. Therefore, by the proof of (ii), we know that ${\rm Cof}(I)={\rm ^\perp \F}$. Thus $({\rm Cof}(I), \F)$ is a complete cotorsion pair by (i).

(iv) This follows from the proofs of (ii) and (iii). \end{proof}

\begin{example}\label{ex:ccpofmodule} Let $R$ be a ring with unit. Let $\A=R\mbox{-}{\rm Mod}$ the category of left $R$-modules. Let $\P$ be the class of all projective $R$-modules. Take $\kappa\geq\aleph_0|R|$, where $|R|$ denotes the cardinality of $R$. Then each $P\in \P$ is $\kappa$-small relative to any class of morphisms in $\A$. Let $I=\{P\to Q\ | \ \forall P, Q\in \P\}$ be a class of monomorphisms of $\A$. Then $I$ is homological and permits the generalized small object argument. Therefore, we have a complete cotorsion pair $({^\perp({\rm Cok}(I)^\perp)}, {\rm Cok}(I)^\perp)$ in $\A$ (note that, in this case, ${\rm Cok}(I)$ consists of $R$-modules with projective dimension $\leq 1$).
\end{example}

\section{Complete cotorsion pairs in the category of chain complexes}

Let $\A$ be an exact category with colimits and pullbacks. Let $\kappa$ be a cardinal. A collection $\P(\kappa)$ of projectives of $\A$ is said to be a class of {\it $\kappa$-small projective generators} if $\P(\kappa)$ is a class of generators of $\A$ and  each object $P\in \P(\kappa)$ is $\kappa$-small relative to split inflations with projective cokernels.

 We denote ${\rm Ch}(\A)$ the category of all chain complexes of objects of $\A$ with the form $\cdots \to  X^{n-1}\xrto{d^{n-1}}  X^n\xrto{d^n} X^{n+1}\to \cdots$ and morphisms are chain complex morphisms. We consider the category ${\rm Ch}(\A)$ as an exact category where conflations degreewise lie in $\A$.
 Given an object $A\in \A$, define $S_n(A)\in {\rm Ch}(\A)$ by $S_n(A)^n=A$ and $S_n(A)^k=0$ for $k\neq n$. Similarly, define $D_n(A)^k=A$ if $k=n$ or $n+1$, but $D_n(A)^k=0$ for other values of $k$, and whose differential is the identity in degree $n$. Note that there is a split conflation $S_{n+1}(A)\to D_n(A)\to S_n(A)$ in ${\rm Ch}(\A)$ for each $n\in \mathbb{Z}$.

  Denote $I=\{S_{n+1}(P)\to D_n(P)\ | \ \forall P\in \P(\kappa), n\in \mathbb{Z} \}$ and $\F={\rm Cok}(I)^\perp$, we have the following corollary.
 \begin{corollary}\label{cor:ccp} $({\rm Cof}(I), \F)$ is a complete cotorsion pair in ${\rm Ch}(\A)$.
\end{corollary}
 \begin{proof} Since $I$ consists of split inflations, we know that each morphism in $I$-cell is also split inflation. By \cite[Lemma 4.3]{Christensen/Hovey}, we know that $S_n(P)$ with $P\in \P(\kappa)$ is $\kappa$-small relative to $I$-cell for all $n\in \mathbb{Z}$. For a morphism $f\colon X\to Y$ in ${\rm Ch}(\A)$ and $i\colon S_{n+1}(P)\to D_n(P)$ in $I$, a morphism from $i$ to $f$ is in bijective correspondence with a commutative diagram
\[\xymatrixcolsep{2pc}\xymatrix@C10pt@R10pt{P\ar[r]\ar[d] &
\ker d^{n+1}_X\ar[d]^{f^{n+1}|_{\ker d^{n+1}_X}}\\
Y^n \ar[r]^{d^n_Y}& Y^{n+1}}
\]
 Let $W^n=\ker d_X^{n+1}\times _{Y^{n+1}}Y^n$ be the pullback. Then there is a deflation frome $P^n$ to $W^n$ such that $P^n$ is a coproduct of objects in $\P(\kappa)$ indexed by a set. Since $P$ is projective in $\A$, the above commutative diagram is in correspondence with a morphism $h^n\colon P\to P^n$. Take $g=\coprod_{n\in \mathbb{Z}}(S_{n+1}(P^n)\to D_n(P^n))$. Then $g$ is in $I$-cell by \cite[Proposition 10.2.7]{Hirschhorn03} and there is an induced morphism from $g$ to $f$ by the morphisms $P^n\to W^n$. By the construction of $g$, the needed factorization property of Definition \ref{def:soa}(iii) holds automatically. Therefore, $I$ permits the generalized small object argument.

Since the exact structure of ${\rm Ch}(\A)$ is given by degreewise conflations in $\A$, we know that $D_n(P)$ is projective in ${\rm Ch}(\A)$ for any $P\in \P(\kappa)$ and $n\in \mathbb{Z}$. From this, it is straightforward to verify $I$ is homological. Moreover, since $\A$ is WIC and ${\rm Cof}(I)$ is a class of generators, then the assertion follows from Theorem \ref{thm:ccp}(iii).
\end{proof}

\begin{example} \label{ex:relativemodelstructure}  Let $\A$ be a complete and cocomplete abelian category with a projective class $\P$ in the sense of \cite[Denfinition 1.1]{Christensen/Hovey}. A short exact sequence $A\xrto{f} B\xrto{g} C$ in $\A$ is called $\P$-exact if $\Hom_\A(P, A)\xrto{f_*} \Hom_\A(P, B)\xrto{g_*} \Hom_\A(P, C)$ is a short exact sequence of abelian groups for any $P\in \P$. The morphism $f$ in a $\P$-exact sequence $A\xrto{f} B\xrto{g} C$ is called a $\P$-inflation and the morphism $g$ is called a $\P$-deflation. Let $\mathcal{E}_\P$ be the class of all $\P$-exact sequences. Then $(\A, \mathcal{E}_\P)$ is an exact category. We will use $\A_\P$ to denote this exact category. Assume that $\P$ has {\it enough $\kappa$-small projectives} in the sense of \cite[Proposition 4.2]{Christensen/Hovey}. I.e., there is a collection $\P'$ of objectives in $\P$ such that every $\P'$-epimorphism is a $\P$-epimorphism and each object in $\P'$ is $\kappa$-small relative to split monomophisms with cokernels in $\P$. Let $I=\{S_{n+1}(P)\to D_n(P)\ | \ \forall P\in \P', n\in \mathbb{Z} \}$ and $\F={\rm Cok}(I)^\perp$. Then $({\rm Ch}(I), \F)$ is a complete cotorsion pair in ${\rm Ch}(\A_\P)$ by Corollary \ref{cor:ccp}.

It can be proved that $\F$ consists of acyclic chain complexes in ${\rm Ch}(\A_\P)$ (a complex $\cdots\to X^n\xrto{d^n} X^{n+1}\xrto{d^{n+1}}X^{n+2}\to \cdots$ is {\it acyclic} if each morphism $X^n\to X^{n+1}$ decomposes in $\A_\P$ as $X^n\xrto{e_n}\ker d^{n+1}\xrto{m_n}X^{n+1}$ where $e_n$ is an deflation and $m_n$ is an inflation; furthermore, $\ker d^{n+1}\xrto{m_n} X^{n+1}\xrto{e_{n+1}} \ker d^{n+2}$ is a conflation). By \cite[Theorem 2.2]{Hovey02}, ${\rm Cof}(I), \F$ and ${\rm Ch}(\A_\P)$ determines a projective closed model structure on ${\rm Ch}(\A_\P)$. It is easy to see that this model structure is Quillen equivalent to the relative closed model structure on ${\rm Ch}(\A)$ given in \cite[Theorem 2.2]{Christensen/Hovey}. So Christensen and Hovey 's relative closed model structure is the standard projective exact closed model structure in this sense.
\end{example}

\begin{example}
(1)  Let $\G$ be a Grothendieck category with a generator $G$. By \cite[Proposition A.2]{Hovey01}, $G$ is small relative to any class of morphisms of $\G$. Let $I=\{S_{n+1}(G)\to D_n(G)\ | \  n\in \mathbb{Z} \}$. Then $({\rm Cof}(I), {\rm Cok}(I)^\perp)$ is a complete cotorsion pair in ${\rm Ch}(\G_G)$ by the above example. This is in fact the \cite[Theorem 4.6]{Gillespie14}.

(2) Let $R$ be a commutative ring. Let $\A$ be the category of $\lambda$-shaped diagrams of $R$-modules with $\lambda$ a small category. A $\lambda$-shaped diagram of sets $T$ is called an {\it orbit} if ${\rm colim}_\lambda T=*$. Let $\mathcal{O}_\lambda$ be the collection of all $\lambda$-orbits. For each orbit $T$, associate a diagram of free $R$-modules $P_T=R(T)$. Let $\P'=\{P_T\ | \ T\in \mathcal{O}_\lambda\}$, and let $\P$ be the projective class determined by $\P'$. Then $\P'$ is a class of $\aleph_0$-small projective generators of the $\P$-exact category $\A_\P$ by \cite[Section 4]{Chorny}. Therefore, $({\rm Cof}(I), {\rm Cok}(I)^\perp)$ is a complete cotorsion pair in ${\rm Ch}(\A_\P)$ by Corollary \ref{cor:ccp}, where $I=\{S_{n+1}(P)\to D_n(P)\ | \ \forall P\in \P', n\in \mathbb{Z} \}$.
\end{example}

\end{document}